\documentclass[11pt,a4paper,english,reqno]{amsart}
\usepackage[a4paper,footskip=1.5em]{geometry}
\usepackage{amsmath,amssymb,amsthm,mathtools}
\usepackage[mathscr]{euscript}
\usepackage[usenames,dvipsnames]{color}
\usepackage{adjustbox,tikz,calc,graphics,babel,standalone}
\usepackage{subcaption}
\usepackage{csquotes,enumerate,verbatim}
\usepackage[final]{microtype}
\usepackage[numbers]{natbib}
\usetikzlibrary{shapes.misc,calc,intersections,patterns,decorations.pathreplacing}
\usepackage{hyperref}
\hypersetup{colorlinks=true,linkcolor=blue,citecolor=blue,pdfpagemode=UseNone,pdfstartview={XYZ null null 1.00}}
\usepackage{cmtiup}

\theoremstyle{plain}
\newtheorem*{theorem*}{Theorem}
\newtheorem{theorem}{Theorem}[section]
\newtheorem{lemma}[theorem]{Lemma}
\newtheorem{claim}[theorem]{Claim}
\newtheorem{proposition}[theorem]{Proposition}
\newtheorem*{claim*}{Claim}

\newtheorem{conjecture}[theorem]{Conjecture}

\theoremstyle{remark}

\newtheorem*{definition}{Definition}

\def\checkmark{\tikz\fill[scale=0.4](0,.35) -- (.25,0) -- (1,.7) -- (.25,.15) -- cycle;}

\newcommand{\C}{\mathcal{C}}
\newcommand{\F}{\mathcal{F}}

\newcommand{\Si}{\Sigma^*}

\newcommand{\ZZ}{\mathbb{Z}_{2^n}}

\let\emptyset\varnothing

\let\originalleft\left
\let\originalright\right
\renewcommand{\left}{\mathopen{}\mathclose\bgroup\originalleft}
\renewcommand{\right}{\aftergroup\egroup\originalright}

\begin{document}

\title{The largest projective cube-free subsets of $\ZZ$}

\author{Jason Long}
\address{Department of Pure Mathematics and Mathematical Statistics, University of Cambridge, Wilberforce Road, Cambridge CB3\thinspace0WB, UK}
\email{jl694@cam.ac.uk}

\author{Adam Zsolt Wagner}
\address{Department of Mathematics, ETH, Z\"urich, Switzerland.}
\email{zsolt.wagner@math.ethz.ch}


\begin{abstract}

In the Boolean lattice, Sperner's, Erd\H{o}s's, Kleitman's and Samotij's theorems state that families that do not contain many chains must have a very specific layered structure. We show that if instead of $\mathbb{Z}_2^n$ we work in $\mathbb{Z}_{2^n}$, several analogous statements hold if one replaces the word \emph{$k$-chain} by \emph{projective cube of dimension $2^{k-1}$}.

We say that $B_d$ is a projective cube of dimension $d$ if there are numbers $a_1, a_2, \ldots, a_d$ such that
$$B_d = \left\{\sum_{i\in I} a_i \bigg\rvert \emptyset \neq I\subseteq [d]\right\}.$$ 

As an analog of Sperner's and Erd\H{o}s's theorems, we show that whenever $d=2^{\ell}$ is a power of two, the largest $d$-cube free set in $\ZZ$ is the union of the largest $\ell$ layers. As an analog of Kleitman's theorem, Samotij and Sudakov asked whether among subsets of $\ZZ$ of given size $M$, the sets that minimize the number of Schur triples (2-cubes) are those that are obtained by filling up the largest layers consecutively. We prove the first non-trivial case where $M=2^{n-1}+1$, and conjecture that the analog of Samotij's theorem also holds.

Several open questions and conjectures are also given.
\end{abstract}
\maketitle

\section{Introduction}\label{sec:intro}

\subsection{Theorems in $\mathbb{Z}_2^n$}

We will consider four important results in the Boolean lattice, which we identify in the usual way with the elements of $\mathbb{Z}_2^n$.

We begin with Sperner's theorem from 1928, a cornerstone result in extremal combinatorics. Here we recall that two distinct sets $A,B\subseteq [n]$ form a $2$-chain if they are comparable, i.e. if $A\subset B$ or $B\subset A$. Similarly, $k$ distinct sets form a $k$-chain if any two of them are comparable.

\begin{theorem}[\label{thm:sperner}Sperner, \cite{sperner}]
If $\F\subset\mathbb{Z}_2^n$ does not contain a $2$-chain then $|\F|$ is not larger than the largest layer.
\end{theorem}

The \emph{layers} of the Boolean lattice are the collections of sets that have the same size, so that the largest layer is $\binom{[n]}{\lfloor n/2 \rfloor}$ and has size $\binom{n}{\lfloor n/2 \rfloor}$. Hence, Sperner's theorem states that a $2$-chain free family has size at most $\binom{n}{\lfloor n/2\rfloor}$. 

One of the many generalisations of Sperner's theorem is due to Erd\H{o}s:

\begin{theorem}[\label{thm:erdos}Erd\H{o}s, \cite{erdos}]
Let $n\geq k\geq 2$ be integers. If $\F\subset\mathbb{Z}_2^n$ does not contain a $k$-chain then $|\F|$ is not larger than the union of the $k-1$ largest layers.
\end{theorem}

Kleitman generalised Sperner's theorem in a different direction. He considered families larger than $\binom{n}{n/2}$ and asked which ones have the fewest $2$-chains. Let the layers of $\{0,1\}^n$ in decreasing order of size be denoted by $J_1,J_2,\ldots, J_{n+1}$ so that $J_1=\binom{[n]}{\lfloor n/2 \rfloor}$ and $\{J_n,J_{n+1}\}=\{\{\emptyset\},\{[n]\}\}$. Say that a family $\F\subseteq \{0,1\}^n$ is \emph{centred} if there exists an $i\in[n+1]$ such that for all $j$ with $1\leq j<i$ we have $J_j\subseteq \F$, and for all $j$ with $i<j\leq n+1$ we have $J_j\cap \F=\emptyset$.

\begin{theorem}[\label{thm:kleitman}Kleitman, \cite{kleitman}]
Let $n\geq 2$ and $M$ be integers. Amongst all families $\F\subseteq \mathbb{Z}_2^n$ of size $|\F|=M$, centred families minimise the number of $2$-chains. 
\end{theorem}

Very recently Samotij, proving a conjecture of Kleitman, generalised Theorems~\ref{thm:erdos} and~\ref{thm:kleitman}.

\begin{theorem}[\label{thm:samotij}Samotij, \cite{samotij}]
Let $n\geq k\geq 2$ and $M$ be integers. Amongst all families $\F\subseteq \mathbb{Z}_2^n$ of size $|\F|=M$, centred families minimise the number of $k$-chains. 
\end{theorem}

The main result of this paper is that most of these results still hold if we replace $\mathbb{Z}_2^n$ by $\mathbb{Z}_{2^n}$. In order to make sense of what the analogs of these theorems are in $\mathbb{Z}_{2^n}$ we need to find analogs of the definition of \emph{layers} and $\emph{chains}$.

\subsection{Theorems in $\mathbb{Z}_{2^n}$}

Finding a natural partition of the elements of $\ZZ$ into $n+1$ layers is not too difficult.  For all $1\leq i\leq n$, let $$L_i:=\{x\in [2^n]: x\equiv 2^{i-1} \text{ }\left(\text{mod }2^i\right)\}$$ be the $i$'th layer and define $L_{n+1}:=\{0\}$. So $L_1$ consists of all odd numbers, $L_2$ is all numbers congruent to $2$ mod $4$, etc. In particular, we have $|L_{i-1}|=2|L_{i}|$ for all $i\leq n$. Finding the right analog of a chain in $\ZZ$ is much more challenging. It was not obvious to us why a corresponding notion should exist, but it turns out it is \emph{projective cubes}.

Following the notation of~\cite{BaLaShWa}, given a multiset $S=\{a_1,\ldots,a_d\}$ of size $d$, we define the \emph{projective $d$-cube} generated by $S$ as 
$$\Si S=\left\{\sum_{i\in I} a_i : \emptyset \neq I\subseteq [k]\right\}.$$
Extremal properties of projective cubes have a vast literature, see e.g.~\cite{cubealon, cubealon2, cubeerdos,cuberodl}. In particular, Rado~\cite{cuberado}, and later independently Sanders \cite{cubesanders} and Folkman (see \cite{cubefolkman} or \cite{cubefolkman2}),
showed that, for any $r$ and $d$, there exists a least number $n$ so that, for any
partition of $[1, n]$ into $r$ classes, one class contains a projective $d$-cube. 

Throughout the rest of this paper we work in the cyclic group $\ZZ$, hence in the definition of $\Si S$ the summations are all modulo $2^n$, and $\Si S$ is a subset of $\ZZ$. Following e.g.~\cite{promel} we do not assume that the numbers $a_1,\ldots,a_k$ are distinct, but we will always view the $d$-cube $\Si S$ as a set, rather than a multiset. Hence $|\Si S|\leq 2^{k}-1$, but $|\Si S|$ could be much smaller. We say that a set $A\subset \ZZ$ is \emph{$d$-cube-free} if there does not exist a multiset $S$ of size $d$ with $\left(\Si S\right)\subseteq A$. 

{\textbf {Examples}}
\begin{itemize}
\item If $S=\{a,b\}$  then $\Si S=\{a,b,a+b\}$ is a $2$-cube (Schur triple). In the degenerate case where $a=b$ we have $\Si S=\{a,2a\}$ which is also a $2$-cube.
\item The set $\{a,2a,3a,\ldots,ka\}$ is a $k$-cube, as it is generated by $S=\{\underbrace{a,\ldots,a}_{k\text{ times}}\}$.
\item If $n=3$ then $\Si \{2,5,5\}=\{2,4,5,7\}$ and hence the set $A=\{2,3,4,5,7\}$ is not $3$-cube-free.
\item The set $\{0\}$ is a $d$-cube for any $d$.
\end{itemize}

We are now ready to state our main results. We obtain the statements by replacing $\mathbb{Z}_2^n$ by $\ZZ$ and the expression ``$k$-chain'' by ``$2^{k-1}$-dimensional projective cube'' in Theorems~\ref{thm:sperner}-\ref{thm:samotij}. We begin with the resulting analog of Sperner's theorem, which is an easy exercise.

\begin{proposition}[\label{prop:schuranal}Analog of Sperner's theorem in $\ZZ$]
If $\F\subset\ZZ$ does not contain a projective $2$-cube then $|\F|$ is not larger than the largest layer, i.e.~$L_1$.
\end{proposition}

Note that a projective $2$-cube is simply a Schur-triple, so Proposition~\ref{prop:schuranal} simply states that any sum-free set in $\ZZ$ has size at most $2^{n-1}$. The analog of Erd\H{o}s' theorem (Theorem~\ref{thm:erdos}) is on largest sets without projective cubes:

\begin{theorem}[\label{thm:powertwo}\label{thm:erdosanal}Analog of Erd\H{o}s' theorem in $\ZZ$]
Let $n\geq k\geq 2$ be  integers. If $\F\subset\ZZ$ does not contain a projective $2^{k-1}$-cube then $|\F|$ is not larger than the union of the $k-1$ largest layers, i.e.~$L_1\cup L_2\cup\ldots\cup L_{k-1}$.
\end{theorem}

This theorem is sharp, since $L_1\cup\ldots\cup L_{k-1}$ is $2^{k-1}$-cube-free (indeed, amongst any collection of $2^{k-1}$ numbers there is a subset whose sum is divisible by $2^{k-1}$). In order to state our version of Kleitman's theorem (Theorem~\ref{thm:kleitman}) we need to define what a \emph{centred} set is. Our definition of centred will be the exact same as in the Boolean lattice case: we say that $S\subset \ZZ$ is \emph{centred} if there exists an $i\in[n+1]$ such that for all $j$ with $1\leq j<i$ we have $L_j\subseteq S$, and for all $j$ with $i<j\leq n+1$ we have $L_j\cap S=\emptyset$. The analog of Kleitman's theorem was raised as a question by Samotij and Sudakov~\cite{samsud} in the very last line of their paper.\footnote{Compare this with Kleitman's conjecture, which was proved by Samotij (see Theorem~\ref{thm:samotij}). This conjecture appeared in the very last line of Kleitman's paper!}

\begin{conjecture}[\label{conj:samsud}Analog of Kleitman's theorem in $\ZZ$]
Let $n\geq 2$ and $M$ be integers. Amongst all families $\F\subseteq \ZZ$ of size $|\F|=M$, centred families minimise the number of $2$-cubes.
\end{conjecture}

For $M\leq 2^{n-1}$ Conjecture~\ref{conj:samsud} is trivial, our modest contribution is that the first non-trivial case of this conjecture is true. Theorem~\ref{thm:kleitmananal} states that the $M=2^{n-1}+1$ case of Conjecture~\ref{conj:samsud} is true. 
\begin{theorem}\label{thm:kleitmananal}
All sets of size $2^{n-1}+1$ in $\ZZ$ contain at least $3\cdot 2^{n-1}$ Schur triples.
\end{theorem}
Following~\cite{samsud}, we define the number of $2$-cubes in a set $A\subseteq \ZZ$ by
$$\mathrm{ST}(A)=\left|\left\{(x,y,z)\in A^3 : x + y = z\right\}\right|,$$
so that if $x+y=z$ and $x\neq y$ then we consider $(x,y,z)$ and $(y,x,z)$ as different triples. Observe that a centred set of size $2^{n-1}+1$, e.g.~the set $L_1\cup\{2\}$, contains precisely $3\cdot 2^{n-1}$ such $2$-cubes, hence Theorem~\ref{thm:kleitmananal} is sharp. The number of $k$-cubes in a set can be defined similarly, and indeed we conjecture that the analog of Samotij's theorem (Theorem~\ref{thm:samotij}) also holds.

\begin{conjecture}[\label{conj:samanal}Analog of Samotij's theorem in $\ZZ$]
Let $n\geq k\geq 2$ and $M$ be integers. Amongst all families $\F\subseteq \ZZ$ of size $|\F|=M$, centred families minimise the number of $2^{k}$-cubes.
\end{conjecture}

\subsection{When $d$ is not a power of two}

While we have seen that $2^{k-1}$-cubes in $\ZZ$ correspond to $k$-chains in $\mathbb{Z}_2^n$, the case of $d$-cubes where $d$ is not a power of two does not seem to have an analog in $\mathbb{Z}_2^n$. Hence it is not obvious what the extremal constructions should be, and indeed this case exhibits a much more interesting behaviour. Table~\ref{table:bestconstr} illustrates our conjectured largest $d$-cube-free constructions, which we refer to as $\mathcal{C}_d$. We will always assume that $n$ is sufficiently large for our constructions to make sense, in particular $n\geq d$ is always enough. Recall that Theorem~\ref{thm:powertwo} establishes that $\C_d$ is indeed best possible for $d=2,4,8,\ldots$.

\begin{table}[b]
\centering
\begin{tabular}{c | c | c } 
 \hline
 $d$& $\mathcal{C}_d$, the largest conjectured $d$-cube-free subset of $\ZZ$ & \\ [0.5ex] 
 \hline\hline
 2 & $ L_1$ &   \checkmark \\ 
\hline
 3 & $L_1  \cup L_3$ &   \\
\hline
4 & $L_1 \cup L_2$ & \checkmark\\
\hline
5 & $L_1 \cup L_2  \cup  L_4$ & \\
\hline
6 & $L_1 \cup L_2 \cup L_4\cup L_6$ & \\
\hline
7  & $L_1 \cup L_2 \cup L_4\cup L_5$ & \\
\hline
8 & $L_1 \cup L_2 \cup L_3$ &\checkmark \\
\hline
9 & $L_1 \cup L_2 \cup L_3\cup L_5$ & \\
\hline
$\ldots$&$\ldots$&\\
\hline
26 & $L_1\cup L_2 \cup L_3 \cup L_4 ~ ~ \cup ~ ~ L_6 \cup L_7 \cup L_8 ~ ~ \cup ~ ~ L_{10}\cup L_{11}$ &\\ [1ex] 
 \hline
 $\ldots$&$\ldots$& \\
 \hline
\end{tabular}
\caption{The conjectured best constructions}
\label{table:bestconstr}
\end{table}

We give an explicit description of this construction $\mathcal{C}_d$ for all $d$ in Section~\ref{sec:notpower}. While we cannot prove that these constructions are best possible (except when $d=2^\ell$), we can show they are best amongst sets that are unions of layers.
\begin{theorem}\label{thm:bestlayer}
Let $d,n$ be positive integers with $d\leq n$. Then $\mathcal{C}_d$ is the largest $d$-cube free subset of $\ZZ$ amongst all sets that can be written as a union of some layers.
\end{theorem}
Our main tool in the proof of Theorem~\ref{thm:bestlayer} is the following elementary lemma, which we believe is interesting in its own right.
\begin{lemma}\label{lem:keylemma}
Let $k\geq 1$ and $x\geq 0$ be integers. Given $2^k+x$ not necessarily distinct integers $a_1,a_2,\ldots,a_{2^k+x}$, at least one of the following two statements holds.
\begin{enumerate}
\item There exists a  subset of these integers whose sum is divisible by $2^{k}$ but not by $2^{k+1}$.
\item There exist $x+1$ disjoint non-empty sets $A_1,\ldots,A_{x+1}\subseteq \{1,2,\ldots,2^k+x\}$, such that for all $s\leq x+1$, we have
$$ \sum_{i\in A_s}a_i\equiv 0 ~ \left(\text{\emph{mod} }2^{k+1}\right).$$
\end{enumerate}
\end{lemma}
The case of $x=0$ in Lemma~\ref{lem:keylemma} follows from the standard statement that amongst $m$ numbers there is a non-empty subset whose sum is divisible by $m$, but already the $x=1$ case is far from trivial. Guaranteeing $x+1$ subsets whose sum is divisible by $2^{k+1}$ in the second point of Lemma~\ref{lem:keylemma} is easy, the power of our lemma lies in the fact that we can take these sets to be disjoint from each other. Our proof of Lemma~\ref{lem:keylemma} relies on a series of compressions and a downward induction on $x$ with base case $x=2^k-1$. Our proof of Lemma~\ref{lem:keylemma} is quite lengthy, it would be very interesting to have a shorter proof.

Given Theorem~\ref{thm:bestlayer}, we would be  surprised if these constructions were not best possible amongst all sets.
\begin{conjecture}\label{conj:cdbestalways}
Let $d,n$ be positive integers with $d\leq n$. Then $\mathcal{C}_d$ is the largest $d$-cube free subset of $\ZZ$.
\end{conjecture}

Our paper is organised as follows. In Section~\ref{sec:power} we focus on the $d=2^\ell$ case and prove Theorem~\ref{thm:powertwo}. The construction $\C_d$ is defined in Section~\ref{sec:notpower}, and there we also prove Lemma~\ref{lem:keylemma} and Theorem~\ref{thm:bestlayer}. Theorem~\ref{thm:kleitmananal}, our partial result on the Samotij-Sudakov question, is proved in Section~\ref{sec:samsud}. Some further open questions and conjectures are given in Section~\ref{sec:outro}.


\section{When $d$ is a power of two}\label{sec:power}
Our main goal in this section is to prove Theorem~\ref{thm:powertwo}. We will prove the following stronger statement, that immediately implies Theorem~\ref{thm:powertwo}.
\begin{theorem}\label{thm:stronger}
Let $\ell,n\in\mathbb{N}^+$ be integers with $2^\ell\leq n$. If $A\subset \ZZ$ satisfies $|A|>\left(1-\frac{1}{2^\ell}\right)2^n,$ then there exist integers $x,y\in [2^n]$ such that $\Si \{\underbrace{x,x,\ldots,x}_{2^\ell-1},y\}\subseteq A$.
\end{theorem}
We will first need the following simple claim.
\begin{claim}\label{claim:ap}
If $A\subset \ZZ$ has size $|A| > \left(1-\frac{1}{2^\ell - 1}\right)2^n$ then there exists an integer $x\in\ZZ$ such that $\left\{x,2x,3x,\ldots,\left(2^\ell-1\right)x\right\}\subseteq A$.
\end{claim}
\begin{proof}
Recall the definition of the layers $(L_i)_{i=1}^{n+1}$ from Section~\ref{sec:intro}. For an integer $1\leq a\leq n$, define the set $\mathcal{F}_a$ as
$$\mathcal{F}_a:=\left\{\left\{x,2x,3x,\ldots,\left(2^\ell-1\right)x\right\} : x\in L_a\right\}.$$
Note that if $a\leq n-\ell +1$ then all elements of $\mathcal{F}_a$ have size exactly $2^\ell - 1$. Indeed, if $i_1x = i_2x$ for some $1\leq i_1 < i_2 \leq 2^\ell - 1$ then $2^n | (i_2-i_1)x$. As $a\leq n-\ell + 1$, we have that $x$ is not divisible by $2^{n-\ell + 1}$, moreover $2^\ell$ cannot divide $i_2-i_1$.

The proof goes by contradiction, let $A$ be a counterexample to the statement of Claim~\ref{claim:ap}. Let $B=\left\{x,2x,3x,\ldots,\left(2^\ell-1\right)x\right\}$ be an element of $\mathcal{F}_a$ and observe that $|B\cap L_a| = 2^{\ell-1}$, $|B\cap L_{a+1}| = 2^{\ell - 2}$, etc, and $|B\cap L_{a+\ell-1}|=1$. Note moreover that every element of $\bigcup_{i=a}^{a+\ell-1}L_i$ appears in precisely $2^{\ell-1}$ different elements of $\mathcal{F}_a$. As for every element $B\in\mathcal{F}_a$ there exists an element $x_B\in B$ with $x_B\notin A$, this implies that $$\frac{|A\cap(L_a\cup L_{a+1}\cup\ldots\cup L_{a+\ell -1})|}{|L_a\cup L_{a+1}\cup\ldots\cup L_{a+\ell -1}|}\leq 1-\frac{1}{2^\ell-1}.$$

Now let $b$ be an integer with $n-\ell+2\leq b \leq n+1$ and observe that since $0\notin A$ (as otherwise we could take $x=0$) we have
$$\frac{|A\cap(L_b\cup L_{b+1}\cup\ldots\cup L_{n+1})|}{|L_b\cup L_{b+1}\cup\ldots\cup L_{n+1}|}\leq 1-\frac{1}{|L_{n-\ell + 2}\cup \ldots\cup L_{n+1}|}=1-\frac{1}{2^{\ell-1}}\leq 1-\frac{1}{2^\ell - 1}.$$
Hence we can partition $\ZZ$ in at most $\lceil (n+1)/\ell\rceil$ parts such that the density of $A$ in each part is at most $1-\frac{1}{2^\ell - 1}$. This completes the proof of Claim~\ref{claim:ap}.
\end{proof}

Now we are ready to give the proof of Theorem~\ref{thm:stronger}. 

\begin{proof}[Proof of Theorem~\ref{thm:stronger}]
Let $\ell,n\in \mathbb{N}^+$ be integers with $2^\ell \leq n$ and let $A\subset \ZZ$ be a set of size $|A|>\left(1-2^{-\ell}\right)2^n$. 
Let $x$ be such that $\{x,2x,3x,\ldots,\left(2^\ell-1\right)x\}\subseteq A$, as guaranteed by Claim~\ref{claim:ap}. Note that as $|A|>\left(1-2^{-\ell}\right)2^n$ we have
$$A'=A\cap (A-x) \cap (A-2x)\cap\ldots\cap \left(A-\left(2^{\ell}-1\right)x\right)\neq \emptyset.$$
Let $y$ be an arbitrary element of $A'$ and note that then we have $y,y+x,y+2x,\ldots,y+\left(2^{\ell}-1\right)x\in A$. Hence we have that $$\left\{x,2x,3x,\ldots,\left(2^\ell-1\right)x,y,y+x,\ldots,y+\left(2^{\ell}-1\right)x\right\}=\Si\{\underbrace{x,x,\ldots,x}_{2^{\ell}-1},y\}\subseteq A$$ and $A$ is not $2^\ell$-cube-free. This completes the proof of Theorem~\ref{thm:stronger}.
\end{proof}


\section{When $d$ is not a power of two}\label{sec:notpower}

Our goal in this section is to define the construction $\C_d$ for all integers $d,n$ with $n$ sufficiently large ($n\geq d$, say), then to prove Lemma~\ref{lem:keylemma} and use it to prove Theorem~\ref{thm:bestlayer}.

\subsection{The construction $\C_d$}

Our conjectured largest $d$-cube free subsets of $\ZZ$ always consist of the union of some of the first few layers, e.g.~$\C_{10}=L_1\cup L_2\cup L_3 \cup L_5 \cup L_7$. Which layers we take does not depend on $n$, as long as the construction makes sense, since e.g.~$L_7$ does not exist if $n=4$. Therefore, when defining $\C_d$ for all $d$ we will always assume that there is enough space in $\ZZ$ for our construction to fit (i.e.~no layers past $L_n$ are included). It will always suffice to take e.g.~$n\geq d$ in general. For positive integers $a,b$ with $a\leq b$ we will use the notation $L_{[a,b]}:=L_a\cup L_{a+1}\cup\ldots\cup L_b$.

We define $\C_d$ recursively as follows.
\begin{enumerate}
\item $\C_1 = \emptyset$.
\item If $d\geq 2$ then let $\ell$ be the largest integer such that $2^\ell \leq d$. Let 
$$\C_d :=L_{[ 1,\ell]} ~ \cup ~ \left\{ 2^{\ell + 1}\cdot x ~ : ~ x \in \C_{d-2^\ell + 1}  \right\}.$$ 
\end{enumerate}
In other words, $\C_d$ is the union of the first $\ell$ layers, it skips $L_{\ell + 1}$, and includes a copy of $\C_{d-2^\ell+1}$ in $L_{\ell + 2}\cup L_{\ell + 3}\cup \ldots$ .

The same definition can be rephrased as follows. For any positive integer $k$, define $\alpha(k)$ to be the largest integer $\ell$ with $2^\ell \leq k$, and let $\beta(k):=k-\alpha(k)+1$. Given $d\geq 2$, set $\ell_1 := \alpha(d)+1$ and let $d_1 := \beta(d)$. Set $\ell_2:=\alpha(d_1)+2$ and let $d_2:=\beta(d_1)$. Repeat until one of the $d_i$-s, say $d_q$, becomes equal to one. We will refer to the resulting sequence $(\ell_1,\ell_2,\ldots,\ell_q)$ as the \emph{block vector} of $\C_d$.  Then we construct $\C_d$ by including the first $\ell_1-1$ layers, \emph{excluding} the next layer, including the next $\ell_2-1$ layers, excluding the layer after these, etc. Hence we get that, letting $M:=\sum_{i=1}^q \ell_i$, $$\C_d=L_{[1,\ell_1-1]}\cup L_{[\ell_1+1,\ell_1+\ell_2-1]}\cup\ldots\cup L_{[M-\ell_q+1,M-1]}.$$

\textbf{Example:}
Suppose we want to find $\C_{26}$. 
\begin{itemize}
\item The largest power of two not greater than $26$ is $2^4=16$. So we include the first four layers $L_1\cup \ldots \cup L_4$ and do not include $L_5$. We replace $26$ by $26-15=11$.
\item The largest  power of two not greater than $11$ is $2^3=8$. Now we include the next three layers $L_6\cup L_7 \cup L_8$ and skip $L_9$. We replace $11$ by $11-7=4$.
\item As $4=2^2$, we include the next two layers $L_{10}\cup L_{11}$. We replace $4$ by $4-3=1$ and stop the algorithm since we hit $1$.
\end{itemize}
So the block vector of $\C_{26}$ is $(5,4,3)$ and  we have $$\C_{26}= L_1\cup L_2 \cup L_3 \cup L_4 ~ ~ \cup ~ ~ L_6 \cup L_7 \cup L_8 ~ ~ \cup ~ ~ L_{10}\cup L_{11}=L_{[1,4]}\cup L_{[6,8]}\cup L_{[10,11]}.$$ We will now use this example and Figure~\ref{fig:c26} to illustrate the intuition behind why this construction is $d$-cube free.

\begin{figure}[h]
\includegraphics{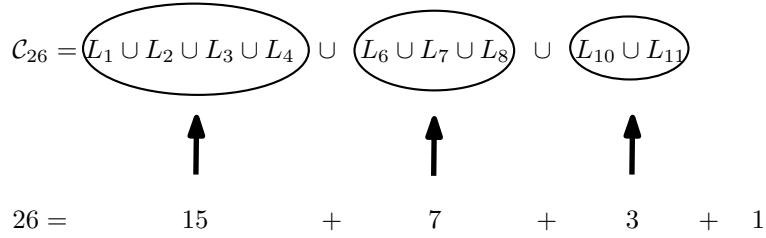}
\caption{$\C_{26}$ is the union of three blocks}
\label{fig:c26}
\end{figure}

Suppose for contradiction that $\C_{26}$ contains a $26$-cube, say $\Si \{x_1,x_2,\ldots,x_{26}\}\subseteq \C_{26}$. Let us call $x_1,\ldots,x_{26}$ the \emph{generators} of the cube. Each of these generators have to lie in either the first block $B_1=L_1\cup L_2\cup L_3\cup L_4$, the second block $B_2=L_6\cup L_7\cup L_8$ or in the third block $B_3=L_{10}\cup L_{11}$. Suppose that precisely $16$ generators lie in $B_1$, $7$ generators lie in $B_2$ and $3$ lie in $B_3$. Consider the $16$ generators lying in $B_1$. The numbers in $B_1$ are all not divisible by $16$, but since we have $16$ generators in $B_1$ we can find a subset sum, say $S$, that is divisible by $16$. As $S\in\Si\{x_1,x_2,\ldots,x_{26}\}\subseteq \C_{26}$ we must have $S\in\C_{26}$, and as $S$ is divisible by $16$ we must have $S\in B_2\cup B_3$. Assume $S\in B_2$. Now the $7$ generators in $B_2$ together with $S$ form $8$ numbers, all divisible by $32$, hence there is a subset sum $S'$ that is divisible by $8\cdot 32$ and thus must be in $B_3$. Now amongst the three generators in $B_3$ together with $S'$ there is a subset sum divisible by $4\cdot 2^9$ and is thus not in $\C_{26}$, which is a contradiction.

The difficulty with making the above intuition rigorous is the following observation. Suppose we are given that $17$ instead of $16$ of the generators lie in $B_1$. Then we can find two sets $S_1,S_2\subset \{x_1,\ldots,x_{26}\}\cap B_1$ such that the sums of elements in $S_1$ and in $S_2$ are both divisible by $16$. The issue is that if $S_1$ and $S_2$ are not disjoint, say they both contain $x_1$, then the number $\sum_{x\in S_1}x + \sum_{x\in S_2}x$ is not necessarily an element of $\Si \{x_1,x_2,\ldots,x_{26}\}$ and hence need not be contained in $\C_{26}$. Luckily Lemma~\ref{lem:keylemma} guarantees that we may take $S_1$ and $S_2$ to be disjoint and the proof goes through.

\subsection{Proof of Theorem~\ref{thm:bestlayer} assuming Lemma~\ref{lem:keylemma}}
The proof consists of two parts. First we use the ideas outlined above, together with Lemma~\ref{lem:keylemma}, to show that $\C_d$ does not contain a $d$-cube. Then we use a simple construction to show that no other set that is a union of layers can be both $d$-cube free and larger than $\C_d$.

\begin{claim}\label{claim:nocube}
For any $d\geq 1$, $\C_d$ does not contain a $d$-cube.
\end{claim}
\begin{proof}
The proof goes by induction on $d$, with $\C_1=\emptyset$ not containing any $1$-cube for any $n\geq d=1$. Let $(\ell_1,\ell_2,\ldots,\ell_q)$ be the block vector of $\C_d$, so that  $\C_d=L_{[1,\ell_1-1]}\cup L_{[\ell_1+1,\ell_1+\ell_2-1]}\cup\ldots\cup L_{[M-\ell_q+1,M-1]},$ where $M=\sum_{i=1}^q \ell_i$. Let the blocks of $\C_d$ be defined in the natural way as $B_1=L_{[1,\ell_1-1]}$, $B_2=L_{[\ell_1+1,\ell_1+\ell_2-1]}$ and in general for $1\leq i \leq q$ we set $B_i=L_{[\ell_1+\ldots +\ell_{i-1}+1,\ell_1+\ldots+\ell_i-1]}$. Assume for contradiction that $\Si\{x_1,\ldots,x_d\} \subset \C_d$. By rearranging we can find an integer $d_1$ such that $x_i\in B_1$ if and only if $i\leq d_1$. The proof splits into two cases according to how large $d_1$ is.

Suppose first that $d_1\leq 2^{\ell_1-1}-1$. Then $\{x_{d_1+1},\ldots,x_d\}\subset \C_d\setminus B_1$ and hence $$\Si\{x_{d_1+1},\ldots,x_d\}\subset \C_d\setminus B_1.$$ Observe that $\{2^\ell_1\cdot x : x\in \C_{d-2^{\ell_1-1}+1}\} = \C_d\setminus B_1$, and hence $$\Si\left\{\frac{x_{d_1+1}}{2^{\ell_1}}, \frac{x_{d_1+2}}{2^{\ell_1}},\ldots,\frac{x_d}{2^{\ell_1}} \right\}\subset \C_{d-2^{\ell_1-1}+1}.$$ This is a contradiction, as $\C_{d-2^{\ell_1-1}+1}$ does not contain a $\left(d-2^{\ell_1-1}+1\right)$-cube.

Hence we must have $d_1\geq 2^{\ell_1-1}$. Applying Lemma~\ref{lem:keylemma} with $k=\ell_1-1$ we conclude that either there is a subset of $\{x_1,\ldots,x_{d_1}\}$ whose sum is divisible by $2^{\ell_1-1}$ but not by $2^{\ell_1}$, or we can find $d_1-2^{\ell_1-1}+1$ disjoint non-empty sets $A_1,\ldots,A_{d_1-2^{\ell_1-1}+1}\subseteq \{1,2,\ldots,d_1\}$ such that for all $s\leq d_1-2^{\ell_1-1}+1$ we have $\sum_{i\in A_s}a_i\equiv 0 ~ \left(\text{mod } 2^{\ell_1}\right)$. The first option is impossible, as $\C_d\cap L_{\ell_1}=\emptyset$, hence the second option must occur. For all $j$ with $1\leq j\leq d_1-2^{\ell_1-1}+1$ let us set $s_j :=\sum_{i\in A_j}a_i $. Then as the $A_i$-s were disjoint, we have that $$\Si\{s_1,s_2,\ldots,s_{d_1-2^{\ell_1-1}+1}, ~ a_{d_1+1},a_{d_2+2},\ldots,a_{d}\}\subset \C_d\setminus B_1.$$ Hence $\C_d\setminus B_1$ contains a $\left(d-2^{\ell_1-1}+1\right)$-cube. As before, this implies that $\C_{d-2^{\ell_1-1}+1}$ contains a $\left(d-2^{\ell_1-1}+1\right)$-cube which is a contradiction. This completes the proof of Claim~\ref{claim:nocube}.
\end{proof}

\begin{claim}\label{claim:cdmax}
For any $d\geq 1$, if $S\subset \ZZ$ is a union of layers and $|S|> |\C_d|$, then $S$ contains a $d$-cube.
\end{claim}
\begin{proof}
As in the proof of Claim~\ref{claim:nocube}, let $(\ell_1,\ell_2,\ldots,\ell_q)$ be the block vector of $\C_d$, so that  $\C_d=L_{[1,\ell_1-1]}\cup L_{[\ell_1+1,\ell_1+\ell_2-1]}\cup\ldots\cup L_{[M-\ell_q+1,M-1]},$ where $M=\sum_{i=1}^q \ell_i$. Let $s$ be index of the first layer wher $\C_d$ and $S$ differ. Because $|L_{i-1}|=2|L_{i}|$ for all $i\leq n$, we must have $L_s\subset S$ and $L_s\cap \C_d=\emptyset$. We will show that $L_{[1,s]}\cap S$ contains a $d$-cube.

For all $i\leq q$ let $M_k:=\sum_{j=1}^i \ell_j$. 
Let $k$ be the largest integer in the set $\{1,2,\ldots,q\}$ that satisfies $M_k \leq s$. We split the proof into two cases, according to whether $s\leq M$ or $s > M$.

If $s\leq M$ then note that $s=M_k$ and we build a multiset $T$ of size $d$, which will be the collection of the generators of the $d$-cube we find in $S$, as follows. First, include $2^{\ell_k}-1$ copies of $2^{M_{k-1}}$ into $T$. Since $L_{[M_{k-1}+1,M_k]}\subset S$, every subset sum of these is in $S$. Next, for all $i\leq k-1$, include $2^{\ell_i-1}-1$ copies of $2^{M_{i-1}}$ into $T$. Given any subset of $T$, the largest power of two dividing its sum is determined by its smallest summands and their quantity, and hence it is easily verified that $\Si T\subset S$. 

By construction $$|T| = \sum_{i=1}^{k-1} \left(2^{\ell_i-1}-1\right)  + 2^{\ell_k}-1.$$ By the definition of the block vector of $\C_d$, we have that $$\alpha\left(d- \sum_{i=1}^{k-1}  \left(2^{\ell_i-1}-1\right)\right) = \ell_k-1.$$ Hence $d- \sum_{i=1}^{k-1}  \left(2^{\ell_i-1}-1\right)\leq 2^{\ell_k}-1$ and so $|T|\geq d$, as required.

The case of $s>M$ is easier, we construct $T$ as follows. For all $i\leq q$ let $T$ contain $2^{\ell_i-1}-1$ copies of $2^{M_{i-1}}$. Moreover, add to $T$ one copy of the number $2^{s-1}$ (which is an element of $L_s$). Checking that $\Si T\subset S$ is similar to the above, and  $|T|=d$ follows from the definition of the block vector of $\C_d$. This finishes the proof of Claim~\ref{claim:cdmax}.
\end{proof}

Note that it is easy to modify the proof of Claim~\ref{claim:cdmax} to show that the constructions $\C_d$ are maximal, i.e.~adding a single element to $\C_d$ makes it not $d$-cube-free. Now Theorem~\ref{thm:bestlayer} follows from Claims~\ref{claim:nocube} and~\ref{claim:cdmax}. It only remains to prove Lemma~\ref{lem:keylemma}, which will take us a significant effort.

\subsection{The proof of Lemma~\ref{lem:keylemma}}

Given a multiset $S=\{a_1,\ldots,a_d\}$ of size $d$, we define the \emph{iterated sumset} of $S$ as 
$$ S^*=\left\{\sum_{i\in I} a_i :I\subseteq [k]\right\},$$
so that $S^*=\left(\Si S\right) \cup \{0\}$.
The reader should be aware that we will be dealing with both sets and multisets in what follows. Anything that is not an iterated sumset is a multiset, also referred to as a collection. Iterated sumsets (of multisets) are just sets, as defined above.
Given a residue $t$ modulo $2^{k+1}$ we define $|t|$ to be the minimal absolute value of any integer in the residue class of $t$ modulo $2^{k+1}$. We will also refer to $|t|$ as the \emph{absolute value} of $t$.
Given an integer $\lambda$ and a multiset $\C$ we define $\lambda\cdot\C=\{\lambda c|c\in \C\}$ (where the RHS is a multiset).

To begin the proof of Lemma~\ref{lem:keylemma}, let $\C$ be a collection of $2^{k}+r$ non-zero residues modulo $2^{k+1}$ with the property that no sub-collection sums to $2^k$ modulo $2^{k+1}$. Then we need to show $\C$ contains at least $r+1$ disjoint, non-empty subsets whose sums are 0 modulo $2^{k+1}$.

The proof will involve two separate ideas. One, which we shall return to later, involves building up the iterated sumset $\C^*$ by introducing elements of $\C$ one by one, and analysing how it can grow. This idea was used by Alon and Freiman~(\cite{cubealon2}, Lemma 4.2) in the following lemma, for which we include their proof.

\begin{lemma}\label{lem:AlonFreiman}[Special case of Alon-Freiman Lemma]
Any collection $\C$ of $2^{k+1}-1$ non-zero residues modulo $2^{k+1}$ contains a non-empty sub-collection summing to $2^k$ modulo $2^{k+1}$.
\end{lemma}
\begin{proof}
The proof considers building up the iterated sumset $\C^*$ by introducing elements of $\C$ one by one. Let $\C=\{c_1,\dots,c_m\}$ where $m=2^{k+1}-1$ and let $\C_i=\{c_1,\dots,c_i\}$. Note that $|\C_1^*|=2$ and $\C_i^*\subseteq \C_{i+1}^*$ for all $i$. If $\C_{i+1}^*=\C_i^*$ for some $i$ then $\C_i^*$ contains the cyclic subgroup of $\ZZ$ generated by $c_{i+1}$. Since $2^k$ is an element of every non-trivial subgroup of $\ZZ$, this completes the proof.
\end{proof}

The second idea is a compression that involves replacing elements of $\C$ in a way that does not change the iterated sumset $\C^*$ or the number of sub-collections that sum to zero. These compressions fall into three categories, which we introduce with the following definition.

\begin{definition}
Let $\C$ be a collection of non-zero residues modulo $2^{k+1}$ with the property that no sub-collection sums to $2^k$ modulo $2^{k+1}$. Suppose that $\C$ contains at least $\lambda>0$ copies of residues $\pm 1$ and also a residue $t$ with $1<|t|\le \lambda+1$. A \emph{type 1 compression} replaces $t$ with $|t|$ copies of the residue 1 if $t\in [1,2^k-1]$ and with $|t|$ copies of $-1$ otherwise. Alternatively, suppose that $\C$ contains a residue $-t$ and two copies of the residue $2^k-t$. Then  replacing the two copies of $2^k-t$ with two copies of $-t$ is called a \emph{type 2 compression}. Finally, if we instead have that $\C$ contains at least $2^{k-1}$ copies of $\pm 1$ and if we have elements $u$ and $v$ that lie in the range $[(3/2)2^{k-1},2^k-1]$ then a \emph{type 3 compression} replaces $u$ and $v$ with $u-2^k$ and $v-2^k$.
\end{definition}

Note that a type 1 compression which replaces the element $t$ increases the number of elements in $\C$ by $|t|-1$. Therefore, provided that we can show that the number of disjoint subsets summing to zero modulo $2^{k+1}$ does not increase by more than $|t|-1$ we will be able to proceed by induction. Type 2 and type 3 compressions do not change the number of elements of $\C$ so we will not be able to immediately apply an induction hypothesis, but provided that we can show that these compressions do not add to the iterated sumset and do not increase the number of disjoint subsets summing to zero modulo $2^{k+1}$ then we will be able to apply them to modify $\C$ in potentially helpful ways. We now prove these properties, justifying our definitions for compressions.

\begin{lemma}\label{lem:compressions}
Let $\C$ be a collection of non-zero residues modulo $2^{k+1}$ with the property that no sub-collection sums to $2^k$ modulo $2^{k+1}$. Let $T_1(\C)$ be the result of applying a type 1 compression to $\C$ (if possible), let $T_2(\C)$ be the result of applying a type 2 compression to $\C$ (if possible) and let $T_3(\C)$ be the result of applying a type 3 compression to $\C$ (if possible). Then $T_1(\C)^*=\C^*$, $T_2(\C)^*\subseteq \C^*$ and $T_3(\C)^*\subseteq \C^*$. Moreover, if $T_2(\C)$ or $T_3(C)$ contain $m$ disjoint sub-collections summing to 0 modulo $2^{k+1}$ then so does $\C$. Lastly, if we write $T_1(C,t)$ for the result of applying a type 1 compression to $\C$ which replaces the element $t$ then we have that if $T_1(\C,t)$ contains $m+|t|-1$ disjoint sub-collections summing to 0 modulo $2^{k+1}$ then $\C$ contains at least $m$ such sub-collections.
\end{lemma}
\begin{proof}
We prove the lemma separately for $T_1(\C)$, $T_2(\C)$ and $T_3(\C)$, starting with $T_1(\C)$.

Without loss of generality we assume that $t\in [2,2^k-1]$ (otherwise we multiply by $-1$). We begin by observing that, given $\alpha+\beta=t-1$, we have
$$\{1^{(\alpha)},-1^{(\beta)}, t\}^*=[-\beta,\alpha]\cup[t-\beta,t+\alpha]$$
and 
$$\{1^{\alpha+t},-1^\beta\}^*=[-\beta,t+\alpha].$$
However, since $\alpha+\beta=t-1$ and so $\alpha = t-\beta-1$ so 
$$[-\beta,\alpha]\cup[t-\beta,t+\alpha]=[-\beta,t+\alpha].$$
It follows that $T_1(\C)^*=\C^*$.

For $T_2$ we also assume without loss of generality that $t\in [1,2^k-1]$. We observe that 
$$\{-t,2^k-t,2^k-t\}^*=\{-t,-2t,-3t,2^k-t,2^k-2t\}$$
and
$$\{-t,-t,-t\}^*=\{-t,-2t,-3t\}\subseteq \{-t,2^k-t,2^k-t\}^*$$
whence $T_2(\C)^*\subseteq \C^*$.

For $T_3$ we observe that, given $\alpha+\beta=2^{k-1}$, we have
$$\{1^{(\alpha)},-1^{(\beta)},u-2^k,v-2^k\}^*$$
$$=[-\beta,\alpha]\cup [u-2^k-\beta, u-2^k+\alpha]\cup [v-2^k-\beta, v-2^k+\alpha]\cup [u+v-\beta, u+v+\alpha].$$
However, since $u$ and $v$ that lie in the range $[(3/2)2^{k-1},2^k-1]$ we have that $u+v$ lies in the range $[-2^{k-1},-2]$ and therefore 
$$[u-2^k-\beta, u-2^k+\alpha]\cup [v-2^k-\beta, v-2^k+\alpha]\subseteq [-\beta,\alpha] \cup [u+v-\beta, u+v+\alpha]$$
and therefore 
$$\{1^{(\alpha)},(-1)^{(\beta)},u-2^k,v-2^k\}^*\subseteq \{1^{(\alpha)},-1^{(\beta)},u+v\}^*\subseteq \{1^{(\alpha)},-1^{(\beta)},u,v\}^*.$$
This shows that $T_3(\C)^*\subseteq \C^*$.

Now suppose that $S_1,\dots, S_{m+t-1}$ are disjoint sub-collections of $T_1(\C,t)$ which all sum to 0 modulo $2^{k+1}$ (as above we assume without loss of generality that $t\in[2,2^k-1]$). Consider the $t$ copies of the element 1 in $T_1(\C,t)$ that result from the type 1 compression replacing $t$. If these do not all appear in distinct $S_i$ then by considering the $S_j$ which do not contain any of these 1s we obtain at least $m$ disjoint sub-collections of $\C$ summing to 0 modulo $2^{k+1}$. If they do all appear in distinct $S_i$ then combining the $S_i$ in which they appear into a single big sub-collection $S$ we may replace the $t$ copies of 1 in $S$ with a copy of $t$ to obtain a sub-collection of $\C$. Combined with the rest of the $S_j$, we obtain $m$ disjoint sub-collections of $\C$ summing to 0 modulo $2^{k+1}$.

In the case of $T_2$ the situation is slightly different. Suppose that $S_1,\dots, S_m$ are disjoint sub-collections of $T_2(\C)$ which all sum to 0 modulo $2^{k+1}$. The type 2 compression replaced two copies of $2^k-t$ with two copies of $-t$. Undoing these replacements in the $S_i$ gives $m$ disjoint sub-collections of $\C$ which all have sum either $0$ or $2^k$ modulo $2^{k+1}$. But since no sub-collection of $\C$ sums to $2^k$ modulo $2^{k+1}$ we get $m$ disjoint sub-collections of $T_2(\C)$ which all have sum $0$.

Lastly, we consider the disjoint sums after a type 3 compression. Again, let $S_1,\dots, S_m$ be disjoint sub-collections of $T_3(\C)$ which all sum to 0 modulo $2^{k+1}$. The type 3 compression replaced $u$ and $v$ with $u-2^k$ and $v-2^k$. As above, we see that undoing these replacements in the $S_i$ gives $m$ disjoint sub-collections of $\C$ which all have sum either $0$ or $2^k$ modulo $2^{k+1}$. But since no sub-collection of $\C$ sums to $2^k$ modulo $2^{k+1}$ we get $m$ disjoint sub-collections of $T_3(\C)$ which all have sum $0$.
\end{proof}

The proof of Lemma~\ref{lem:keylemma} will proceed by induction on both $k$ and $r$. The induction on $r$ will proceed downwards from $r=2^k-2$ -- more details on this will follow. In light of Lemma~\ref{lem:compressions} we will be able to use our induction hypothesis on $r$ if we are ever able to apply a type 1 compression (to $\C$ itself or to any $\lambda \cdot \C$ where $\lambda$ is odd). We will therefore care about the properties of `maximally type 1 compressed' collections.

To this end, we return to the idea presented in Lemma~\ref{lem:AlonFreiman}. For our purposes the idea will need to be extended a little, requiring a more detailed analysis of how the iterated sumset of $\C$ can grow whilst avoiding the residue $2^k$ under the additional assumption that no $\lambda \cdot \C$ for $\lambda$ odd can be type 1 compressed. This process is captured in the following lemma.

\begin{lemma}\label{lem:main}
Let $k\ge 3$ and $r\ge 1$. Let $\C$ be a collection of $2^{k}+r$ non-zero residues modulo $2^{k+1}$ with the property that no sub-collection sums to $2^k$ modulo $2^{k+1}$. Assume also that no $\lambda \cdot \C$ for $\lambda$ odd can be type 1 compressed. Then either $\C$ contains $2^{k-1}+r$ even residues, or there is an odd residue $t$ (modulo $2^{k+1}$) such that $\C$ contains $2^{k-1}+r$ residues which are either $\pm t$ or $\pm (2^k-t)$.
\end{lemma}

In order to prove this lemma, we shall need two technical lemmas that analyse the process of building $\C^*$ by taking into account new elements of $\C$ one by one.

\begin{lemma}\label{lem:yuck1}
Let $\C$ be a collection of $2^{k}+r$ non-zero residues modulo $2^{k+1}$ with the property that no sub-collection sums to $2^k$ modulo $2^{k+1}$. Assume also that no $\lambda \cdot \C$ for $\lambda$ odd can be type 1 compressed. For some fixed $i$ let $\C_i\subset \C$ with $|\C_i|=i$. Now choose $\C_{i+1}=C_i\cup\lbrace x_j\rbrace$ for some $x_j\in \C\setminus \C_i$ so that $\C_{i+1}^*\setminus\C_{i}^*$ is maximal. We claim that $|\C_{i+1}^*\setminus\C_{i}^*|$ is greater than 2 unless one of the following cases holds:
\begin{enumerate}
\item $|\C_i^*|\le 5$ or $|C_i^*|\ge 2^{k+1}-5$.
\item All elements of $\C \setminus \C_i$ are even.
\item All elements of $\C\setminus\C_i$ are either $\pm u$ or $\pm(2^k-u)$ modulo $2^{k+1}$ for some odd $u$.
\end{enumerate}
\end{lemma}
\begin{proof}
Suppose we reach a point where all remaining generators increase the sumset by at most 2. Let our pool of remaining generators be called $T=\C\setminus \C_i$, and our iterated sumset so far is $\C_i^*$. Assume that $|\C_i^*|\ge 6$ or we are in case (1).

If there are no odd elements in $T$ then all remaining elements are even and we are in case (2). So assume that there is an odd generator in $T$. Without loss of generality (by multiplying everything by some odd residue $\lambda$) we may assume this generator is $\pm 1$. Our goal is now to show that all remaining generators are $\pm 1$ or $\pm (2^k-1)$.

Since by assumption we know that including the $\pm 1$ increases the iterated sumset size by at most 2, we have that $\C_i^*$ is a union of two intervals $I_1$ and $I_2$ (of course it could be a single interval, which is also a union of two intervals). 

Now we consider the possibilities for other elements in $T$. Suppose $a\in T$. We wish to show $a=\pm 1$ or $\pm (2^k-1)$. 

We have that $|((I_1+a)\cup(I_2+a))\setminus(I_1\cup I_2)|\le 2$. Clearly this allows $a$ to be equal to $\pm 1$. If $a$ is $\pm 2$ then we are done because we can do a type 1 compression and replace the $\pm 2$ with two $\pm 1$s. As usual, we may assume that $a\in[1,2^k-1]$ without loss of generality, by multiplying everything by $-1$. 

Suppose $I_1+a$ intersects $I_1$ but does not intersect $I_2$. Then since $|(I_1+a)\setminus I_1|\le 2$ we have $a\le 2$ and so $a=1$ (since $a=2$ is forbidden), or $|I_1|\le 2$. But the latter case also implies $a=1$ since $I_1+a$ intersects $I_1$. Similarly, if $I_2+a$ intersects $I_2$ but does not intersect $I_1$ then we get that $a=1$.

For a proper interval $I$ we denote the $x\in I$ such that $x+1 \not\in I$ by $M(I)$ and the $x\in I$ such that $x-1 \not\in I$ by $m(I)$.

By the above, if $I_1+a$ intersects $I_1$ we may assume that $I_1+a$ also intersects $I_2$. So $I_1+a$ contains the entire gap $I_{g}$ between $M(I_1)$ and $m(I_2)$ which must therefore have size at most 2. So now let $I$ be the whole interval consisting of $I_1$, $I_2$ and $I_g$. Note that $I+a$ contains at least $\min(a,4)$ new elements that do not belong in $I$, since $|I|\le 2^{k+1}-6+|I_g|\le 2^{k+1}-4$ by assumption. But 
$$2\ge|((I_1+a)\cup(I_2+a))\setminus(I_1\cup I_2)| \ge |(I+a)\setminus I|-|I_g|+|I_g|$$
since at most $|I_g|$ elements from $(I+a)\setminus I$ do not belong to $((I_1+a)\cup(I_2+a))$ and also $I_g\subseteq ((I_1+a)\cup(I_2+a))$. So $a\le 2$ and therefore $a=1$ (since $a=2$ is forbidden). 

Therefore we may assume that $I_1+a$ is disjoint from $I_1$ and $I_2+a$ is disjoint from $I_2$.

Observe that 
$$|((I_1+a)\cup(I_2+a))\setminus (I_1\cup I_2)|=|(I_1+a)\setminus I_2|+|(I_2+a)\setminus I_1|$$
since $I_1$ and $I_2$ are disjoint.

We have that 
$$|(I_1+2a)\setminus I_1|\le |(I_2+a)\setminus I_1|+|(I_1+2a)\setminus (I_2+a)|=  |(I_2+a)\setminus I_1|+|(I_1+a)\setminus I_2| $$
so
$$2\ge |(I_1+a)\setminus I_2|+|(I_2+a)\setminus I_1|$$
$$=|(I_1+2a)\setminus I_1|.$$

Similarly $|(I_2+2a)\setminus I_2|\le 2.$ Since $|I_1|+|I_2|\ge 6$ we may assume without loss of generality that $|I_1|\ge 3$. But since $|(I_1+2a)\setminus I_1|\leq 2$, we deduce that $2a=\pm 1$ or $2a=\pm 2$, so $a$ is $\pm 1$ or $\pm (2^k-1)$.
\end{proof}

We need one final technical lemma that controls the size of $\C_3$.

\begin{lemma}\label{lem:yuck2}
Let $\C$ be a collection of $2^{k}+r$ non-zero residues modulo $2^{k+1}$ with the property that no sub-collection sums to $2^k$ modulo $2^{k+1}$. Assume also that no $\lambda \cdot \C$ for $\lambda$ odd can be type 1 compressed. Either $\C$ contains $2^{k-1}+r$ even residues, or there is an odd residue $t$ (modulo $2^{k+1}$) such that $\C$ contains $2^{k-1}+r$ residues which are either $\pm t$ or $\pm (2^k-t)$, or we can find 3 elements $\{x_1,x_2,x_3\}$ in $\C$ so that $|\{x_1,x_2,x_3\}^*|\ge 6.$
\end{lemma}
\begin{proof}
Suppose that we can find $x_1,x_2,x_3$ odd and distinct in $\C$. Then $\{x_1,x_2,x_3\}^*$ contains $\{x_1,x_2,x_3,x_1+x_2,x_1+x_3,x_2+x_3\}$ which are all distinct and we are done.

Suppose there are at most $2^{k-1}$ even elements in $\C$. Then we may now assume that the odd residues, of which there are at least $2^{k-1}+r$, all lie in at most two residue classes, say $u$ and $v$, and without loss of generality $u$ appears at least twice. If $u\neq v$ then 
$$\{0,u,v,2u,u+v,2u+v\}\subseteq \{u,u,v\}^*$$
and so $|\{u,u,v\}^*|\ge 6$ unless $u=2^k\pm v$ or $u=\pm v$. In this case there is an odd residue $t$ (modulo $2^{k+1}$) so that $u$ and $v$ are both either $\pm t$ or $\pm (2^k-t)$ so we are done.

We may now assume that $\C$ contains more than $2^{k-1}$ even residues, so in particular it contains a repeated even residue $u$. Note also that if $\C$ does not contain $2^{k-1}+r$ even residues then it contains at least $2^{k-1}$ odd residues. Let $v$ be an odd element of $\C$. Then $\{0,u,2u,v,u+v,2u+v\}\subseteq \{u,u,v\}^*$ and these are all distinct, so $|\{u,u,v\}^*|\ge 6$.
\end{proof}

We are now ready to prove Lemma~\ref{lem:main}.

\begin{proof}[Proof of Lemma~\ref{lem:main}]

By Lemma~\ref{lem:yuck2} we are either immediately done or we may find $\C_3$ of size 3 such that $|\C_3^*|\ge 6$. Now inductively define $\C_{i+1}=C_i\cup\{x\}$ for some choice of $x$ in $\C\setminus \C_i$ so that $\C_{i+1}^*-\C_{i}^*$ is maximal. Let $j\ge 3$ be minimal so that $|\C_{j+1}^*\setminus \C_j|\le 2$. Then $|\C_j^*|\ge 3(j-3)+6=3j-3.$ Note that we always have that $|\C_{i+1}^*-\C_{i}^*|\ge 1$, since, by Lemma~\ref{lem:AlonFreiman}, the size of the iterated sumset increases by at least 1 whenever a new element is introduced. Since $|\C^*|\le 2^{k+1}-1$ we have
$$3j-3+2^k+r-j\le 2^{k+1}-1$$
which implies that
$$j\le 2^{k-1}-\Big\lceil\frac{r}{2}\Big\rceil+1$$
and so
$$2^k+r-j \ge 2^{k-1}+\Big\lceil\frac{3r}{2}\Big\rceil-1\ge 2^{k-1}+r.$$
This is the number of remaining elements in $\C$. By Lemma~\ref{lem:yuck1} we may deduce that either 
\begin{enumerate}
\item $|\C_i^*|\le 5$ or $|C_i^*|\ge 2^{k+1}-5$.
\item All remaining elements of $\C$ are even.
\item All remaining elements of $\C$ are either $\pm u$ or $\pm(2^k-u)$ modulo $2^{k+1}$ for $u$ odd.
\end{enumerate}
In case (2) or (3) we are done since the number of remaining elements is at least $2^{k-1}+r$. Note also that $|\C_i^*|\ge |C_3^*|\ge 6$, while if $|C_i^*|\ge 2^{k+1}-5$ then we must have $2^{k-1}+r\le 4$ since each remaining element increases the size of the iterated sumset by 1 and $|\C^*|\le 2^{k+1}-1$, but $k\ge 3$ and $r\ge 1$ so this is impossible.
\end{proof}

We are finally ready to prove Lemma~\ref{lem:keylemma}.

\begin{proof}[Proof of Lemma~\ref{lem:keylemma}]
We prove the result by induction on $k$ and $r$. The cases $k\le 2$ are a trivial check. The induction goes downwards on $r$ starting at a base case $r=2^k-2$. The case $r=0$ is also done separately. These base cases are covered in detail as follows.
\begin{enumerate}
\item $k\le 2$: This requires checking that given $4+r$ residues modulo $8$ that avoid a sum of $4$ modulo $8$ then we can find $r+1$ sums which are $0$ modulo $8$, since the $k=1$ case has $0\le r\le 2-2=0$ so is trivial. The above check for $k=2$ can be done by hand.
\item $r=0$: This follows immediately from the standard result that among $n$ numbers there is a sum which is $0$ modulo $n$. In our case we have $2^k$ numbers in $\C$ so there is a non-trivial sub-collection with sum 0 modulo $2^k$. Since no sum is $2^k$ modulo $2^{k+1}$ the sum is in fact $0$ modulo $2^{k+1}$.
\item $r=2^k-2$: Observe that in this case we have $|\C|=2^{k+1}-2$ and $|\C^*|=2^{k+1}-1$. Therefore if we list the elements of $\C=\{x_1,\dots,x_s\}$ and define $\C_i=\{x_1,\dots,x_i\}$ then we must have that $|\C_i^*|=i+1$ for every $i$. In particular, $|\C_{i+1}^*\setminus C_i^*|=1$ for each $i$. In particular, $C_2^*$ has size 3. This holds for any ordering of the $x_i$, but if $u\neq \pm v$ then $|\{u,v\}^*|=4$ so we find that all $x_i$ are $\pm t$ for some residue $t$. Since $|\C|=2^{k+1}-2$ we find that $t$ cannot be even (otherwise all elements of $\C$ are even and so $|\C^*|\le 2^{k-1}$ which is impossible) so by scaling we may assume all elements of $\C$ are $\pm 1$. The only possibility that avoids a sum of $2^k$ is to have $2^k-1$ copies of $+1$ and $2^k-1$ copies of $-1$. In this case we get $2^k-1=r+1$ disjoint pairs which sum to 0 (by pairing up the $+1$s and $-1$s).
\end{enumerate}
For the inductive case we may apply Lemma~\ref{lem:main}. We find that either $\C$ contains $2^{k-1}+r$ even residues, or there is an odd residue $t$ (modulo $2^{k+1}$) such that $\C$ contains $2^{k-1}+r$ residues which are either $\pm t$ or $\pm (2^k-t)$. In the first case we may simply divide everything by 2 -- we are left with a collection $\C/2$ of $2^{k-1}+r$ residues modulo $2^k$ which avoid a sum of $2^{k-1}$ modulo $2^k$. By induction on $k$ this contains $r+1$ disjoint sums which are $0$ modulo $2^{k}$, and the corresponding disjoint sums in $\C$ are all $0$ modulo $2^{k+1}$.

So we are left in the case where there is an odd residue $t$ (modulo $2^{k+1}$) such that $\C$ contains $2^{k-1}+r$ residues which are either $\pm t$ or $\pm (2^k-t)$. Observe that having $+t$ and $+2^k-t$ is impossible as it gives a sum of $2^k$, and similarly for $-t$ and $-(2^k-t)$. 

Without loss of generality we assume that $\pm t$ occurs at least once. Then we can use type 2 compressions to obtain that all but at most one of the $2^{k-1}+r$ residues are $\pm t$. By scaling we can assume that $t=1$. We have thus reduced to the case where $\C$ contains at least $2^{k-1}+r-1\ge 2^{k-1}$ (as if $r=0$ then done) copies of $\pm 1$. Since we are done by induction on $r$ if we are able to perform any type 1 compressions, we deduce that $\C$ does not contain any residues $u$ with $|u|\in [2,2^{k-1}]$.

Therefore all remaining generators have absolute value in the range $[2^{k-1}+1,2^{k}-1]$. Note that if two members $u,v$ of $\C$ lie in the range $[(3/2)2^{k-1}+1,2^{k}-1]$ then we can do a type 3 compression to replace $u$ and $v$ with $u-2^k$ and $v-2^k$ respectively. Then since $u-2^k$ and $v-2^k$ have absolute value less than $2^{k-1}$ we can do a type 1 compression and we are done by induction unless $u-2^k=v-2^k=-1$, ie $u=v=2^k-1$. But after our type 2 compressions we had that all but at most one of the elements equal to $\pm 1$ or $\pm (2^k-1)$ were in fact $\pm 1$, so this cannot arise. Thus if we can find such $u$ and $v$ then we are done.

Similarly, if we find two members $u,v$ of $\C$ lie in the range $[2^{k}+1, 2^k + 2^{k-2}-1]$ then we can multiply by $-1$ to get two elements in the range $[(3/2)2^{k-1}+1,2^{k}-1]$. We are then done as above.

If we find $u\in [(3/2)2^{k-1}+1,2^{k}-1]$ and $v\in [2^{k}+1, 2^k + 2^{k-2}-1]$ then we must have a sum congruent to $2^k$ modulo $2^{k+1}$, since we have at least $2^{k-1}$ elements which are $\pm 1$.

We may therefore assume that $\C$ contains at most one element with absolute value in the range $[(3/2)2^{k-1}+1,2^{k}-1]$.

The last range to consider is the elements of $\C$ with absolute value in the range $[2^{k-1}+1,(3/2)2^{k-1}-1]$. Suppose we have elements $u$ and $v$ in $\C$ with absolute value in this range. Then order the elements of $\C$ as follows. Let $x_1=u$ and $x_2=v$. Then let $x_3,\dots,x_{2^{k-1}+2}=\pm 1$. Then take the remaining elements in any order. Let $\C_i=\{x_1,\dots,x_i\}$. Note that for $i\le 2^{k-1}+2$, we have
$$\C_i^*=[u-\beta,u+\alpha]\cup[-\beta,\alpha]\cup [u+v-\beta,u+v+\alpha]$$
where there are $\alpha$ copies of $+1$ and $\beta$ copies of $-1$ amongst $x_3,\dots,x_{2^{k-1}+2}$. By the absolute values of $u$ and $v$, we see that these 3 intervals are disjoint. So $|\C_{2^{k-1}+2}^*|\ge 3(2^{k-1}+1)$. 

Since adding each further element of $\C$ increases the size of the iterated sumset by at least 1, we get that 
$$|\C^*|\ge 3(2^{k-1}+1)+2^k+r-(2^{k-1}+2)\geq 2^{k+1}$$
which is impossible.

So we conclude that at most 1 element of $\C$ has absolute value in the range $[2^{k-1}+1,(3/2)2^{k-1}-1]$, and thus at most 2 elements of $\C$ have absolute value in the range $[2^{k-1}+1, 2^{k}-1]$. All other members of $\C$ must be equal to $\pm 1$ or we can do type 1 compressions.

This means that we in fact have at least $2^k+r-2\ge 2^k-1$ elements of $\C$ which are $\pm 1$. But this means that we can do type 1 compression if \emph{any} element of $\C$ is not equal to $\pm 1$. Since at most $2^k-1$ of these can be $+1$, we get at least $2^k+r-(2^k-1)=r+1$ disjoint pairs of $\{\pm 1\}$, giving us $r+1$ disjoint sums equalling 0. This finishes the proof of Lemma~\ref{lem:keylemma}.
\end{proof}


\section{Families minimising the number of $2$-cubes}\label{sec:samsud}

Our goal in this section is to prove Theorem~\ref{thm:kleitmananal}. Given $n$, let $S\subset \ZZ$ be a set of size $M=2^{n-1}+1$. Recall that for $M\leq 2^{n-1}$ the centred family of size $M$ is a subset of $L_1$ and hence contains no Schur triples, our $M=2^{n-1}+1$ case is the first non-trivial case of Conjecture~\ref{conj:samsud}. Our goal is to prove that any family of size $2^{n-1}+1$ contains at least $3\cdot 2^{n-1}$ Schur triples, which is the number of Schur triples in the centred set $S=L_1\cup\{2\}$. We will use for all $i$ the notation $S_i = S\cap L_i$ where $L_i$ is the $i$'th layer as before, and similarly $S_{i+} = S\cap \left(L_{i+1}\cup L_{i+2}\cup \ldots\cup L_{n+1}\right)=S_{i+1}\cup S_{i+2}\cup\ldots\cup S_{n+1}$. Note that whenever the numbers $x,y,z$ form  a Schur triple, they cannot be in three different layers, nor all in the same layer.

Denote by $C(a,b,c)$ the set of Schur triples $(x,y,z)\in S^3$ with $x\in S_a$, $y\in S_b$ and $z\in S_c$. Similarly, let e.g.~$C(a+,b,c):=\left\{(x,y,z)\in S^3: ~  x+y=z, ~ x\in S_{a+}, y\in S_b, z\in S_c\right\}$. We will use the following elementary observation on the number of edges in an induced subgraph of a regular graph.

\begin{claim}\label{claim:regugraph}
Let $G$ be a directed graph, with bidirectional edges (i.e.~$x\rightarrow y$ and $y\rightarrow x$) and loops allowed. Suppose every vertex has out-degree and in-degree equal to $k$, let $N:=|V(G)|$ and let $R\subset V(G)$ be a set of size $|R|=m$. Then the number of edges in the induced subgraph $G[R]$ satisfies
$$E(G[R])\geq \max\{m(k-N+m),k(2m-N),0\}.$$
\end{claim}
\begin{proof}
Since every vertex in $R$ has $k$ edges leaving it, and at most $N-m$ of these end in vertices not in $R$, it follows that at least $k-N+m$ point to vertices in $R$, and the first part follows. The middle inequality follows from the observation that there are $km$ edges starting at vertices of $R$ and there are $(N-m)k$ edges ending at vertices not in $R$. The third part holds since the number of edges cannot be negative.
\end{proof}

\begin{claim}\label{claim:firstsam}
For any integer $a$ with $1\leq a\leq n$, we have $$|C{(a,a,a+)}|\geq \max\left\{ |S_a|\left(|S_{a+}| - |L_a|+|S_a|\right),|S_{a+}|(2|S_a|-|L_a|),0\right\}.$$
\end{claim}
\begin{proof}
Let $z\in S_{a+}$. Create a directed graph $G_z$ on vertex set $L_a$ by adding the edge $x\rightarrow y$ if $x+y=z$. Every vertex in this graph has indegree and outdegree one (loops possible) and every edge corresponds to a Schur triple in $C(a,a,a+)$. Let $$G=\bigcup_{z\in S_{a+}} G_z,$$ so that $G$ is a directed graph with every vertex having indegree and outdegree exactly $|S_{a+}|$, and every directed edge present at most once, and the bound follows from Claim~\ref{claim:regugraph}. When $S=T$ we either have $|S_{a+}|=0$ or $S_a=L_a$, with equality in both cases.
\end{proof}

\begin{claim}\label{claim:secondsam}
For any integer $a$ with $1\leq a\leq n$, we have $$|C{(a+,a,a)}|\geq \max\left\{ |S_a|\left(|S_{a+}| - |L_a|+|S_a|\right),|S_{a+}|(2|S_a|-|L_a|),0\right\}.$$
\end{claim}
\begin{proof}
The proof is very similar to the proof of Claim~\ref{claim:firstsam}. Fix an element  $y\in S_{a+}$ and create a directed graph $G_y$ on vertex set $L_a$ by adding the edge $x\rightarrow z$ if $y=z-x$. Then in $G=\cup_{y\in S_{a+}} G_y$ every vertex has indegree and outdegree $|S_{a+}|$, and the rest of the proof is exactly as in Claim~\ref{claim:firstsam}. When $S=T$ we have equality as before.
\end{proof}

\begin{proof}[Proof of Theorem~\ref{thm:kleitmananal}]
Recall that $\mathrm{ST}(S)$ denotes the number of Schur triples in $S$. By Claims~\ref{claim:firstsam} and~\ref{claim:secondsam} we have
\begin{equation*}
\begin{split}
\mathrm{ST}(S)&=\sum_{a=1}^n C(a+,a,a) + C(a,a+,a) + C(a,a,a+) \\
&\geq 3 \max\left\{ |S_a|\left(|S_{a+}| - |L_a|+|S_a|\right),|S_{a+}|(2|S_a|-|L_a|),0\right\}.
\end{split}
\end{equation*}
It suffices to show that amongst sets of size $M$ the function $$f(S)=3\sum_{a=1}^n \max\left\{ |S_a|\left(|S_{a+}| - |L_a|+|S_a|\right),|S_{a+}|(2|S_a|-|L_a|),0\right\}$$ is never less than $3|L_1|$.

For every element $x\in\ZZ$, define $g(x)$ to be the integer satisfying $x\in L_{g(x)}$. Observe that if there exists an integer $b\geq 1$ such that $0<|S_{b+}|\leq |L_b\setminus S_b|$ then we can replace all of $S_{b+}$ by arbitrary elements of $L_b\setminus S_b$. This does not increase $f(S)$ and decreases $\sum_{x\in S} g(x)$. Setting $B:=\max\{i: 0<|S_i|\}$, the highest non-empty layer, we can assume that for all $b<B$ we have $|L_b\setminus S_b|<|S_{b+}|$. So we have the strict inequalities

\begin{equation*}
\begin{split}
|L_1|&<|S_1| + |S_2| + \ldots + |S_B| = |L_1|+1\\
|L_2|&< |S_2| + |S_3| + \ldots + |S_B| \\
|L_3|&< |S_3| + |S_4| + \ldots + |S_B|\\
\ldots&\\
|L_{B-1}|&< |S_{B-1}| + |S_B| .
\end{split}
\end{equation*}
Hence we have
\begin{equation*}
\begin{split}
\frac{f(S)}{3}&\geq \left(\sum_{a=1}^{B-2} |S_a|(|S_{a}| + |S_{a+}| - |L_a|)\right) ~ + ~ |S_B|(2|S_{B-1}| - |L_{B-1}|)\\
&\geq \left(\sum_{a=1}^{B-2} |S_a|\right) ~ + ~  |S_B|(2|S_{B-1}| - |L_{B-1}|)
\end{split}
\end{equation*}
Now note that subject to the constraints $0\leq |S_{B-1}|\leq |L_{B-1}|$, $0\leq |S_{B}|\leq |L_B|$ and $|S_{B-1}|+|S_{B}|\geq |L_{B-1}|+1$, we have the inequality $|S_{B}|(2|S_{B-1}|-|L_{B-1}|) \geq (|S_{B-1}| + |S_B| - |L_{B-1}|)|L_{B-1}|$. So we have
\begin{equation*}
\begin{split}
\frac{f(S)}{3}&\geq \left(\sum_{a=1}^{B-2} |S_a|\right) ~ + ~ (|S_{B-1}| + |S_B| - |L_{B-1}|)|L_{B-1}|\\
&\geq |L_1|+1 -|S_{B-1}| - |S_{B}| + (|S_{B-1}| + |S_B| - |L_{B-1}|)|L_{B-1}|\\
&\geq |L_1| + 1 - (|L_{B-1}| + 1) + |L_{B-1}| = |L_1|,
\end{split}
\end{equation*}
where in the last inequality we used that $|S_{B-1}| + |S_B|>|L_{B-1}|$. Hence $f(S)\geq 3|L_1|$ and this completes the proof of Theorem~\ref{thm:kleitmananal}.
\end{proof}


\section{Conclusion and open questions}\label{sec:outro}

The main conjectures raised in the present paper are the analog of Samotij's theorem in $\ZZ$ (Conjecture~\ref{conj:samanal}) and that the constructions $\C_d$ are best possible (Conjecture~\ref{conj:cdbestalways}). The first open case of Conjecture~\ref{conj:cdbestalways} is the $d=3$ case, which we restate here.

\begin{conjecture}\label{conj:conc1}
The largest $3$-cube-free family in $\ZZ$ has size $(5/8)\cdot 2^n$. 
\end{conjecture}
Recall that $L_1\cup L_3$ is $3$-cube-free, so this conjecture, if true, is sharp. Using Gurobi~\cite{gurobi} we could check that for $n\leq 7$ the following stronger conjecture is also true.
\begin{conjecture}\label{conj:conc2}
If $S\subset \ZZ$ is a set of size $|S|>(5/8)\cdot 2^n$ then there exist $x,y$ such that $\Si\{x,x,x\}\subset A$ or $\Si\{x,3x,y\}\subset A$.
\end{conjecture}
If true, Conjecture~\ref{conj:conc2} might be easier to prove than Conjecture~\ref{conj:conc1}. Our hope is that an insightful proof of Conjecture~\ref{conj:conc1} may quickly lead to a full proof of Conjecture~\ref{conj:cdbestalways}. 

Following~\cite{alonconc,erdosconc,guy}, a natural complementary problem is to determine for all $M,n$ (and $k$) the set $S\subset \ZZ$ of size $|S|=M$ with the largest number of Schur triples (or $k$-chains). In the Boolean lattice for a wide range of $M$ the constructions with the largest number of comparable pairs are essentially towers of cubes. In $\ZZ$, say that a set $S\subset \ZZ$ is \emph{anti-centred} if  there exists an $i\in[n+1]$ such that for all $j$ with $n\geq j>i$ we have $L_j\subseteq S$, and for all $j$ with $i>j\geq 1$ we have $L_j\cap S=\emptyset$. It seems plausible that anti-centred families maximize the number of Schur triples and perhaps even $k$-cubes. Note that if $M=2^{n-\ell}$ for some $\ell$ then an anti-centred family is the union of the $n-\ell+1$ smallest layers of $\ZZ$ and hence contains $M^2$ Schur triples and $M^{k}$ distinct $k$-cubes, both of which are optimal.

Another related problem, following~\cite{satu2,satu}, is to determine the smallest maximal $k$-cube free set in $\ZZ$. That is, the smallest $S\subset \ZZ$ that is $k$-cube free, but the additon of any new element to $S$ makes it not $k$-cube-free. Here we do not even have a good guess about what the extremal families could be.

\section*{Acknowledgements}
We are very grateful to Anton Bernshteyn for many stimulating conversations throughout the project.

\end{document}